\documentclass[a4paper,12pt,reqno]{amsart}

\usepackage{amsfonts}
\usepackage{amsmath}
\usepackage{amssymb}

\usepackage{mathrsfs}
\usepackage{hyperref}

\numberwithin{equation}{section}

\usepackage{graphicx}

\newtheorem{theorem}{Theorem}[section]
\newtheorem{defi}[theorem]{Definition}
\newtheorem{remark}[theorem]{Remark}
\newtheorem{corollary}[theorem]{Corollary}

\def\Rn{{\mathbb R}^n}
\def\R2n{{\mathbb R}^{2n}}

\def\Rn{{\mathbb R}^n}

\def\R2{{\mathbb R}^2}
\def\R2n{{\mathbb R}^{2n}}

\def\HS{{\mathtt{HS}}}
\def\N0{{\mathbb N}_{0}}
\def\onu{{1/\nu}}
\def\paal{{\partial^{\alpha}}}

\begin{document}

\title[Eigenfunction expansions of ultradifferentiable functions]
{Eigenfunction expansions of ultradifferentiable functions and 
ultradistributions}
\author[Aparajita Dasgupta]{Aparajita Dasgupta}

\address{
  Aparajita Dasgupta:
  \endgraf
  \'Ecole polytechnique f\'ed\'erale de Lausanne
  \endgraf
  Facult\'e des Sciences
  \endgraf
  CH-1015 Lausanne 
  \endgraf
  Switzerland
  \endgraf
  {\it E-mail address} {\rm aparajita.dasgupta@epfl.ch}
  }

\author[Michael Ruzhansky]{Michael Ruzhansky}
\address{
  Michael Ruzhansky:
  \endgraf
  Department of Mathematics
  \endgraf
  Imperial College London
  \endgraf
  180 Queen's Gate, London SW7 2AZ 
  \endgraf
  United Kingdom
  \endgraf
  {\it E-mail address} {\rm m.ruzhansky@imperial.ac.uk}
  }
  
  \dedicatory{Dedicated to the memory of Todor Gramchev (1956-2015)}

\thanks{The second
 author was supported by the EPSRC Grant EP/K039407/1 and by 
 the Leverhulme Research Grant RPG-2014-02.}
\date{\today}

\subjclass{Primary 46F05; Secondary 22E30}
\keywords{Gevrey spaces; ultradistributions; Komatsu classes.}

\begin{abstract}
In this paper we give a global characterisation of classes of ultradifferentiable
functions and corresponding ultradistributions on a compact manifold $X$.
The characterisation is given in terms of the eigenfunction expansion of an 
elliptic operator on $X$. This extends the result for analytic functions on
compact manifolds by Seeley \cite{see:exp}, and the characterisation
of Gevrey functions and
Gevrey ultradistributions on compact Lie groups and homogeneous spaces by 
the authors \cite{Dasgupta-Ruzhansky:BSM}.
\end{abstract}

\maketitle

\section{Introduction}

Let $X$ be a compact analytic manifold and let $E$ be an analytic, elliptic, positive differential 
operator of order $\nu.$ 
Let $\{\phi_j\}$ and $\{\lambda_j\}$ be respectively the eigenfunctions and eigenvalues of
$E,$ i.e. $E\phi_j=\lambda_j\phi_j.$
Then acting by $E$ on a smooth function $$f=\sum_{j}f_j\phi_j$$ we see that it is 
analytic if and only if that there is a constant $C>0$ such that for all $k\geq 0$
$$
\sum_{j} \lambda_j^{2k}|f_j|^2\leq ((\nu k)!)^2C^{2k+2}.
$$ 
Consequently, Seeley has shown in \cite{see:exp} that 
$f=\sum_{j}f_j\phi_j$ is analytic if and only if 
the sequence $\{A^{\lambda_j^{\onu}}f_j\}$ is bounded for some $A>1$.

The aim of this paper is to extend Seeley's characterisation to classes more general than that
of analytic functions. In particular, the characterisation we obtain will cover Gevrey spaces
$\gamma^{s}$, $s\geq 1$,
extending also the characterisation that was obtained previously by the authors in
\cite{Dasgupta-Ruzhansky:BSM} in the setting of compact Lie groups and compact
homogeneous spaces. The characterisation will be given in terms of the behaviour
of coefficients of the series expansion of functions with respect to the eigenfunctions
of an elliptic positive pseudo-differential operator $E$ on $X$, similar to 
Seeley's result in \cite{see:exp} (which corresponds to the case $s=1$), 
with a related construction in
\cite{see:ex}.
Interestingly, our approach allows one to define and analyse analytic or Gevrey functions
even if the manifold $X$ is `only' smooth. It also applies to quasi-analytic,
non-quasi-analytic, and other classes of functions.

We  also analyse dual spaces,
compared to Seeley who restricted his analysis to analytic spaces only
(and this will require the analysis of so-called $\alpha$-duals).

Global characterisations as obtained in this paper have several applications.
For example, the Gevrey spaces appear naturally when dealing with weakly 
hyperbolic problems, such as the wave equation for sums of squares of
vector fields satisfying H\"ormander's condition, also with the time-dependent
propagation speed of low regularity. In the setting of compact Lie groups 
the global characterisation of Gevrey spaces that has been obtained
by the authors in \cite{Dasgupta-Ruzhansky:BSM} has been further applied
to the well-posedness of Cauchy problems associated to sums of squares 
of vector fields in \cite{Garetto-Ruzhansky:wave-eq}. In this setting the Gevrey
spaces appear already in $\Rn$ and come up naturally in energy
inequalities. There are other applications, in particular in the theory of
partial differential equations, see e.g. Rodino \cite{Rodino:bk-Gevrey}.

More generally, our argument will give a characterisation
of functions in Komatsu type classes resembling but more general than those
introduced by Komatsu in
\cite{Komatsu:ultra-I,Komatsu:ultra-II}.
Consequently, we give a characterisation of the corresponding
dual spaces of ultradistributions.
We discuss Roumieu and Beurling type (or injective and projective limit, respectively)
spaces as well as their duals and $\alpha$-duals in the sense of K\"othe
\cite{Kothe:BK-top-vector-spaces-I} which also turn out to be perfect spaces.  

In the periodic setting (or in the setting of functions on the torus) different function spaces
have been intensively studied in terms of their Fourier coefficients.
Thus, periodic Gevrey functions have been discussed in terms of their Fourier coefficients
by Taguchi \cite{Taguchi:torus-YMJ-1987}
(this was recently extended to general compact Lie groups and homogeneous
spaces by the authors in \cite{Dasgupta-Ruzhansky:BSM}), see also
Delcroix, Hasler, Pilipovi\'c and Valmorin \cite{Delcroix-Hasler-Pilipovic:periodic}
for the periodic setting. 

More general periodic Komatsu-type classes have been considered by
Gorbachuk \cite{Gorbachuk:Komatsu-UMJ-1982}, with tensor product structure
and nuclearity properties analysed by Petzsche \cite{Petzsche:MM-1979}.
See also Pilipovi\'c and Prangoski \cite{Pilipovic-Prangoski:Roumieu-MM-2014}
for the relation to their convolution properties. Regularity properties of spaces
of ultradistributions have been studied by 
Pilipovi\'c and Scarpalezos \cite{Pilipovic-Scarpalezos:PAMS-2001}. We can refer to
a relatively recent book by 
Carmichael, Kami\'nski and Pilipovi\'c \cite{Carmichael-Kamiski-Pilipovic:BK}
for more details on these and other properties of ultradistributions on 
$\Rn$ and related references.

The paper is organised as follows. In Definition \ref{DEF:Komatsu} we introduce the class
$\Gamma_{\{M_k\}}(X)$ of ultradifferentiable functions on a manifold $X$, and in
Theorem \ref{PROP:Komatsu} we give several alternative reformulations and further
properties of these spaces. 
In Theorem \ref{THM:Komatsu} we characterise these classes in terms of
eigenvalues of a positive elliptic pseudo-differential operator $E$ on $X$.
Consequently, we show that since
the spaces of Gevrey functions $\gamma^{s}$ on $X$
correspond to the choice $M_{k}=(k!)^{s}$, in Corollary \ref{COR:Gevrey} 
we obtain a global characterisation for Gevrey spaces $\gamma^{s}(X)$.
These results are proved in Section \ref{SEC:proofs}.
Furthermore, in Theorem \ref{THM:ultra} we describe the eigenfunction
expansions for the corresponding spaces of ultradistributions.
This is achieved by first characterising the $\alpha$-dual spaces in
Section \ref{SEC:alpha} and then relating the $\alpha$-duals to
topological duals in Section \ref{SEC:ultradistributions}. 
In Theorem \ref{THM:ultra} we describe the counterparts of the obtained results for
the Beurling-type spaces of ultradifferentiable functions and ultradistributions.

We will denote by $C$ constants, taking different values sometimes even in the
same formula. We also denote $\N0:=\mathbb N\cup\{0\}$.

\section{Formulation of results on compact manifolds}
\label{SEC:Komatsu}

Let $X$ be a compact $C^{\infty}$ manifold of dimension $n$ without boundary and with a fixed measure. 
We fix $E$ to be a positive elliptic pseudo-differential operator of an integer
order $\nu\in\mathbb N$ on $X$, and we 
write $$E\in\Psi^{\nu}_{+e}(X)$$ in this case. 
For convenience we will assume that $E$ is classical
(although this assumption is not necessary).
The eigenvalues of $E$ form a sequence 
$\{\lambda_{j}\}$, and for each eigenvalue $\lambda_j$ we denote the 
corresponding eigenspace by $H_j$. 
We may assume that $\lambda_{j}$'s are ordered as
$$0<\lambda_{1}<\lambda_{2}<\cdots.$$
The space $H_{j}\subset L^{2}(X)$ consists of 
smooth functions due to the ellipticity of $E.$ 
We set
$$
d_j:=\dim H_j, \; H_0:=\ker E, \; \lambda_0:=0, \; d_0:=\dim H_0.
$$
Since the operator $E$ is elliptic, it is Fredholm, hence also $d_0<\infty.$
It can be shown (see \cite[Proposition 2.3]{Delgado-Ruzhansky:JFA-2014}) that
\begin{equation}\label{EQ:dl1}
d_{j}\leq C(1+\lambda_{j})^{\frac{n}{\nu}}
\end{equation}
for all $j$, and that
\begin{equation}\label{EQ:dl2}
 \sum_{j=1}^{\infty} d_{j} (1+\lambda_{j})^{-q}<\infty
 \quad \textrm{ if and only if } \quad 
 q>\frac{n}{\nu}.
\end{equation}
We denote by $e_{j}^{k}$, $1\leq k\leq d_{j}$, an orthonormal basis in $H_{j}$.
For $f\in L^{2}(X)$, we denote its `Fourier coefficients' by
$$
 \widehat{f}(j,k):=(f,e_{j}^{k})_{L^{2}}.
$$
We will write 
\begin{equation}
\widehat{f}(j) :=\begin{pmatrix}
    \widehat{f}(j,1)\\
    \vdots\\
    \vdots\\
    \widehat{f}(j,d_j)
  \end{pmatrix}
  \in
{\mathbb C}^{d_{j}},
\end{equation}
for the whole Fourier coefficient corresponding to $H_{j}$,
and we can then write
$$
 \|\widehat{f}(j)\|_{\HS}= (\widehat{f}(j) \cdot \widehat{f}(j))^{1/2}=
 \left( \sum_{k=1}^{d_{j}} |\widehat{f}(j,k)|^{2}\right)^{1/2}.
$$
We note the Plancherel formula
\begin{equation}\label{EQ:Plancherel}
\|f\|_{L^{2}(X)}^{2}=\sum_{j=0}^{\infty} \sum_{k=1}^{d_{j}} |\widehat{f}(j,k)|^{2}
=\sum_{j=0}^{\infty}  \|\widehat{f}(j)\|_{\HS}^{2}.
\end{equation}
We refer to \cite{Delgado-Ruzhansky:JFA-2014}, and for further developments in
\cite{Delgado-Ruzhansky:invariant} for a discussion of different properties of the 
associated Fourier analysis. Here we note that the ellipticity of $E$ and the 
Plancherel formula imply the 
following characterisation of smooth functions in terms of their Fourier 
coefficients:
\begin{equation}\label{EQ:smooth}
f\in C^{\infty}(X) \; \Longleftrightarrow \;
\forall N\; \exists C_{N}: 
|\widehat{f}(j,k)|\leq C_{N} \lambda_{j}^{-N} 
\textrm{ for all } j\geq 1, 1\leq k\leq d_{j}.
\end{equation}
If $X$ and $E$ are analytic, the result of Seeley \cite{see:exp} can be reformulated 
(we will give a short argument for it in the proof of Corollary \ref{COR:Gevrey})
as
\begin{equation}\label{EQ:analytic}
f \textrm{ is analytic } \Longleftrightarrow 
 \exists L>0\; \exists C: 
|\widehat{f}(j,k)|\leq C e^{-L\lambda_j^{1/\nu}} 
\textrm{ for all } j\geq 1, 1\leq k\leq d_{j}.
\end{equation}
Similarly, we can put $\|\widehat{f}(j)\|_{\HS}$ on the left-hand sides of these inequalities.
The first aim of this note is to provide a characterisation similar to \eqref{EQ:smooth} and
\eqref{EQ:analytic} for classes of functions in between smooth and analytic functions,
namely, for Gevrey functions, and for their dual spaces of (ultra)distributions.
However, the proof works for more general classes than those of Gevrey functions,
namely, for classes of functions considered by Komatsu \cite{Komatsu:ultra-I}, as well
as their extensions described below.

We now define an analogue of these classes in our setting.
Let $\{M_k\}$ be a sequence of positive numbers such that 

\begin{enumerate}
\item[(M.0)] $M_0 = 1$
%\item[(M.1)] (logarithmic convexity) $M_{k}^2\leq M_{k-1}M_{k+1}, ~~~~~~~k=1,2,3,\dots$
\item[(M.1)]  (stability) $$M_{k+1}\leq AH^{k}M_{k},~~~~k=0,1,2,\ldots$$
%\item[(M.2)] there exists $C>0$ such that $M_{k}\geq C k!$ for  
%all $k\in\mathbb N.$
%there exists $k_{0}$ such that the sequence $\{M_{k}\}_{k\geq k_{0}}$ is increasing.
\end{enumerate}

%The assumption (M.2) will guarantee that the classes of function that we will
%consider are well-defined on analytic manifolds. 

For our characterisation of functional spaces we will be also assuming condition 
\begin{enumerate}
\item[(M.2)]  
$$M_{2k}\leq AH^{2k} M_k^{2},~~~~~k=0,1,2,\ldots$$
\end{enumerate}
We note that conditions (M.1)$+$(M.2) are a weaker version of the condition
assumed by Komatsu, namely, the condition
$$
M_k\leq AH^{k}\min_{0\leq q\leq k}M_q M_{k-q}, k=0,1,2,\ldots,
$$
which ensures the stability under the application of ultradifferential operators. However, we note that the above condition can be shown to be actually equivalent to (M.2), see 
\cite[Lemma 5.3]{Petzsche: MA-1984}. The assumptions (M.0), (M.1) and (M.2) are weaker than those imposed by 
Komatsu in \cite{Komatsu:ultra-I,Komatsu:ultra-II} who was also assuming
%, instead of condition (M.3)
%was assuming a stronger condition 
\begin{enumerate}
\item[(M.3$^{\prime}$)]  $\sum_{k=1}^{\infty}\frac{M_{k-1}}{M_k}<\infty.$
\end{enumerate}
Thus, in \cite{Komatsu:ultra-I,Komatsu:ultra-II,Komatsu:ultra-III} 
Komatsu investigated classes of
ultradifferentiable functions on $\Rn$ associated to the sequences $\{M_{k}\}$,
namely, the space of functions $\psi\in C^{\infty}(\Rn)$ such that
for every compact $K\subset\Rn$ there exist $h>0$ and a constant $C$ such that
\begin{equation}\label{EQ:Kom-Rn}
\sup_{x\in K} |\partial^{\alpha}\psi(x)|\leq C h^{|\alpha|} M_{|\alpha|}.
\end{equation}
Sometimes one also assumes logarithmic convexity, i.e. the condition
\begin{enumerate}
\item[] (logarithmic convexity) $\;$ $M_{k}^2\leq M_{k-1}M_{k+1}, ~~~~~~~k=1,2,3,\dots$
\end{enumerate}
This is useful but not restrictive, namely, one can always find an alternative
collection of $M_{k}$'s defining the same class but  satisfying the 
logarithmic convexity condition, see e.g. Rudin \cite[19.6]{Rudin:bk-RandCanalysis-1974}.
We also refer to Rudin \cite{Rudin:bk-RandCanalysis-1974} and to
Komatsu \cite{Komatsu:ultra-I} for examples of different classes satisfying
\eqref{EQ:Kom-Rn}. These include analytic and Gevrey functions, quasi-analytic and
non-quasi-analytic functions (characterised by the Denjoy-Carleman theorem),
and many others.

Given a space of ultradifferentiable functions satisfying \eqref{EQ:Kom-Rn}
we can define a space of ultradistributions as its dual.
Then, among other things, in \cite{Komatsu:ultra-II} Komatsu showed that
under the assumptions (M.0), logarithmic convexity, (M.1) and (M.3$^{\prime}$),
$f$ is an ultradistribution supported in $K\subset\Rn$ if and only if
there exist $L$ and $C$ such that
\begin{equation}\label{EQ:Komatsu-Rn}
|\widetilde{f}(\xi)|\leq C\exp(M(L\xi)),\quad \xi\in\Rn,
\end{equation}
and in addition for each $\epsilon>0$ there is a constant $C_{\epsilon}$ such that
\begin{equation}\label{EQ:Komatsu-Rn-2}
|\widetilde{f}(\zeta)|\leq C_{\epsilon}\exp(H_{K}(\zeta)+\epsilon|\zeta|),\quad \zeta\in\mathbb C^{n},
\end{equation}
where $\widetilde{f}(\zeta)=\langle e^{-i\zeta x}, f(x)\rangle$ 
is the Fourier-Laplace transform of $f$,
$M(r)=\sup_{k\in \mathbb N}\log\frac{r^k}{M_k}$,
 and 
$H_{K}(\zeta)=\sup_{x\in K} {\rm Im} \langle x,\zeta\rangle.$
There are other versions of these estimates given by, for example,
Roumieu \cite{Roumieu:1962} or Neymark \cite{Neymark:1969}. Moreover, by
further strengthening assumptions (M.1), (M.2)
and (M.3$^{\prime}$) one can prove a version
of these conditions without the term $\epsilon|\zeta|$ in \eqref{EQ:Komatsu-Rn-2},
see \cite[Theorem 1.1]{Komatsu:ultra-II} for the precise formulation.

We now give an analogue of Komatsu's definition on a compact $C^{\infty}$ manifold $X$.

\begin{defi} \label{DEF:Komatsu}
The class $\Gamma_{\{M_k\}}(X)$ is the space of $C^{\infty}$ functions $\phi$ on $X$ 
such that there exist $h>0$ and $C>0$ such that we have
\begin{equation}\label{EQ:def-Komatsu-L2}
\|E^{k}\phi\|_{L^{2}(X)}\leq Ch^{\nu k}M_{\nu k},~~~~~k=0,1,2,\ldots
\end{equation}
\end{defi}
We can make several remarks concerning this definition.
\begin{remark}
{\rm (1)} Taking $L^{2}$-norms in \eqref{EQ:def-Komatsu-L2} is
convenient for a number of reasons. But we can already note that 
by embedding theorems and properties of 
the sequence $\{M_{k}\}$ this is equivalent to taking $L^{\infty}$-norms,
or to evaluating the corresponding action of a frame of vector fields 
(instead of the action of powers of a single operator $E$) on functions,
see Theorem \ref{PROP:Komatsu}.

{\rm (2)} This is also equivalent to classes of functions belonging to the corresponding
function spaces in local coordinate charts, 
see Theorem \ref{PROP:Komatsu}, {\rm (v)}.
In order to ensure that we cover the cases
of analytic and Gevrey functions we will be assuming that $X$ and $E$
are analytic.

{\rm (3)} The advantage of Definition \ref{DEF:Komatsu} is that we do not refer to
local coordinates to introduce the class $\Gamma_{\{M_k\}}(X)$. This allows for definition
of analogues of analytic or Gevrey functions even if the manifold $X$ is `only' smooth.
For example, by taking $M_{k}=k!$, we obtain the class $\Gamma_{\{k!\}}(X)$
of functions satisfying the condition
\begin{equation}\label{EQ:an}
\|E^{k}\phi\|_{L^{2}(X)}\leq Ch^{\nu k}(\nu k)!,~~~~~k=0,1,2,\ldots
\end{equation}
If $X$ and $E$ are analytic, we will show in Corollary \ref{COR:Gevrey} that this is
precisely the class of analytic functions\footnote{
An alternative argument could be
to use Theorem \ref{THM:Komatsu} for the characterisation of this class in terms of
the eigenvalues of $E$, and then apply Seeley's result \cite{see:exp} showing that this
is the analytic class.}
on $X$. However, if $X$ is `only' smooth, the space of locally analytic functions does not
make sense while definition given by \eqref{EQ:an} still does. We also note that such 
space $\Gamma_{\{k!\}}(X)$ is still meaningful, for example it contains constants (if $E$ is a 
differential operator), as well as the eigenfunctions of the operator $E$.
\end{remark}

In the sequel we will be also always assuming that 
\begin{align*}
k!\leq C_{l} l^{k}M_{k}, \quad \forall k\in\mathbb{N}_{0} \quad &(\mbox{Roumieu case: for some }  l,C_{l}>0)
\\&
(\mbox{Beurling case: for all } l>0 \mbox{ there is } C_{l}>0).
\end{align*}’
We summarise properties
of the space ${\Gamma_{\{M_k\}}(X)}$ as follows:

\begin{theorem}\label{PROP:Komatsu} 
%{\rm (M.0), (M.1)} and {\rm (M.2)}.
We have the following properties:
\begin{itemize}
\item[(i)]
The space ${\Gamma_{\{M_k\}}(X)}$ is independent of the choice of an operator
$E\in\Psi^{\nu}_{e}(X)$, i.e. $\phi\in{\Gamma_{\{M_k\}}(X)}$ if and only if 
\eqref{EQ:def-Komatsu-L2} holds for one (and hence for all) elliptic 
pseudo-differential operators $E\in\Psi^{\nu}_{e}(X)$.
\item[(ii)]
We have $\phi\in{\Gamma_{\{M_k\}}(X)}$ if and only if there exist constants 
$h>0$ and $C>0$ such that
\begin{equation}\label{EQ:def-Komatsu-Linf}
\|E^{k}\phi\|_{L^{\infty}(X)}\leq Ch^{\nu k}M_{\nu k}, ~~~~~~~k=0,1,2,\ldots
\end{equation}
\item[(iii)]
Let $\partial_{1},\cdots,\partial_{N}$ be a frame of smooth vector fields
on $X$ (so that $\sum_{j=1}^{N} \partial_{j}^{2}$ is elliptic).
%, normalised such that $|\partial_{j}|\leq 1$ with respect to some
%Riemannian structure on $X$, $j=1,\ldots,N$.
Then $\phi\in \Gamma_{\{M_k\}}(X)$ if and only if 
there exist $h>0$ and $C>0$ such that
\begin{equation}\label{EQ:def-vf}
\|\partial^{\alpha}\phi\|_{L^{\infty}(X)}\leq Ch^{|\alpha|}M_{|\alpha|},
\end{equation}
for all multi-indices $\alpha$, where
$\partial^{\alpha}=\partial_{j_{1}}^{\alpha_{1}}\cdots \partial_{j_{K}}^{\alpha_{K}}$
with $1\leq j_{1},\cdots,j_{K}\leq N$ and 
$|\alpha|=\alpha_{1}+\cdots+\alpha_{K}.$
\item[(iv)]
We have $\phi\in \Gamma_{\{M_k\}}(X)$ if and only if 
there exist $h>0$ and $C>0$ such that
\begin{equation}\label{EQ:def-vf2}
\|\partial^{\alpha}\phi\|_{L^{2}(X)}\leq Ch^{|\alpha|}M_{|\alpha|},
\end{equation}
for all multi-indices $\alpha$ as in {\rm (iii)}.
\item[(v)]
Assume that $X$ and $E$ are analytic.
Then the class $\Gamma_{\{M_k\}}(X)$
is preserved by analytic changes of variables, and hence is well-defined on $X$.
Moreover, in every local coordinate chart, it consists of functions locally belonging
to the class
$\Gamma_{\{M_k\}}(\Rn)$.
\end{itemize}
\end{theorem}

The important example of the situation in Theorem \ref{PROP:Komatsu} , Part (v), is 
that of Gevrey classes $\gamma^{s}(X)$ 
of ultradifferentiable functions, when we have
$\gamma^{s}(X)=\Gamma_{\{M_k\}}(X)$ with the constants $M_{k}=(k!)^{s}$ for $s\geq 1.$
By Theorem \ref{PROP:Komatsu}, Part (v), this is the space of Gevrey functions on
$X$, i.e. functions which belong to the Gevrey classes $\gamma^{s}(\Rn)$ in all
local coordinate charts, i.e. such that there exist $h>0$ and $C>0,$ such that
\begin{equation*}\label{EQ:def-vfRn}
\|\partial^{\alpha}\psi\|_{L^{\infty}(\Rn)}\leq Ch^{|\alpha|}(|\alpha|!)^{s},
\end{equation*}
for all localisations $\psi$ of the function $\phi$ on $X$, and for all multi-indices $\alpha$.

If $s=1$, this is the space of analytic functions.

\medskip
For the sequence $\{M_k\}$, we define the \textit{associated function} as
$$
M(r):=\sup_{k\in \mathbb N}\log\frac{r^{\nu k}}{M_{\nu k}},\quad r>0,
$$ 
and we may set $M(0):=0$.
We briefly establish a simple property of eigenvalues $\lambda_{l}$ useful
in the sequel, namely that for every $q$, $L>0$ and $\delta>0$ there exists $C>0$
such that we have
\begin{equation}\label{EQ:est-ev}
\lambda_l^{q}e^{-\delta M\left(L\lambda_l^{\onu}\right)}\leq
C\quad \textrm{uniformly in } l\geq 1. 
\end{equation}
Indeed, from the definition of the function $M$ it follows that 
$$
\lambda_l^{q}e^{-\delta M\left(L\lambda_l^{\onu}\right)}\leq 
\lambda_l^{q}\frac{M^{\delta}_{\nu p}}{L^{\nu p\delta}\lambda_l^{p\delta}},
$$
for every $p\in \mathbb N$.
In particular, using this with $p$ such that $p\delta =q+1$, we obtain
$$
\lambda_l^{q}e^{-\delta M\left(L\lambda_l^{\onu}\right)}\leq
\frac{M^{\delta}_{\nu p}}{L^{\nu(q+1)}\lambda_l}\leq C
$$
uniformly in $l\geq 1$, implying \eqref{EQ:est-ev}.

Now we characterise the class ${\Gamma_{\{M_k\}}(X)}$ of ultradifferentiable functions
in terms of eigenvalues of operator $E$. Unless stated explicitly, we usually
assume that $X$ and $E$ are only smooth (i.e. not necessarily analytic).

\begin{theorem}\label{THM:Komatsu}
Assume conditions {\rm (M.0), (M.1), (M.2).} 
Then $\phi\in{\Gamma_{\{M_k\}}(X)}$ if and only if there exist constants 
$C>0,$ $L>0$ such that 
$$
\|\widehat{\phi}(l)\|_{\HS}\leq C\exp\{-M(L \lambda_l^{\onu} )\}
\quad \textrm{ for all } \; l\geq 1,
$$ 
where $M(r)=\sup_{k}\log\frac{r^{\nu k}}{M_{\nu k}}.$ 
\end{theorem}

For the Gevrey class 
$$\gamma^{s}(X)=\Gamma_{\{(k!)^{s}\}}(X),\quad 1\leq s<\infty,$$ 
of 
(Gevrey-Roumieu) ultradifferentiable functions
we have $M(r)\simeq r^{1/s}$.
Indeed, using the inequality $\frac{t^N}{N!}\leq e^{t}$ we get
\begin{eqnarray*}
M(r)=\sup_{k}\log\frac{r^{\nu k}}{((\nu k)!)^s}\leq 
\sup_{\ell}\log\left[\frac{(r^{1/s})^\ell}{\ell!}\right]^{s}
\leq \sup_{\ell}\log\left[e^{sr^{1/s}}\right]
= sr^{1/s}.
\end{eqnarray*}
On the other hand, first we note the inequality
$$
\inf_{p\in\mathbb N} (2p)^{2ps} r^{-2p} \leq \exp\left(-\frac{s}{8e} r^{1/s}\right),
$$
see \cite[Formula (3.20)]{Dasgupta-Ruzhansky:BSM} for the proof.
Using the inequality $(k+\nu)^{k+\nu}\leq (4\nu)^{k} k^{k}$ for $k\geq \nu$, 
an analogous proof yields the inequality 
$$
\inf_{p\in \mathbb N} (\nu p)^{\nu ps} r^{-\nu p} \leq 
\exp\left(-\frac{s}{4\nu e} r^{1/s}\right).
$$
This and the inequality $p!\leq p^{p}$ imply
\begin{multline*}
\exp(M(r))=\sup_{k}\frac{r^{\nu k}}{((\nu k)!)^s}
= \frac{1}{\inf_{k}\{(r^{1/s})^{-\nu k}((\nu k)!)\}^{s}} \\
\geq 
\frac{1}{\inf_{k}\{(r^{1/s})^{-\nu k}(\nu k)^{\nu k}\}^{s}} 
%\geq 
%\frac{1}{\inf_{\nu k}\{(r^{1/s})^{-\nu k}(\nu k)^{\nu k}\}^{s}}
\geq \frac{1}{\exp\{-(\frac{s}{4\nu e})r^{1/s}\}}
=\exp{{\left(\frac{s}{4\nu e}r^{1/s}\right)}}.
\end{multline*}
Combining both inequalities we get
\begin{equation}\label{EQ:M-Gevrey}
\frac{s}{4\nu e}r^{1/s} \leq M(r)\leq sr^{1/s}.
\end{equation}
Consequently, we obtain the characterisation of Gevrey spaces:

\begin{corollary}\label{COR:Gevrey}
Let $X$ and $E$ be analytic and let $s\geq 1$.
We have $\phi\in \gamma^{s}(X)$ if and only if there exist constants 
$C>0,$ $L>0$ such that 
$$
\|\widehat{\phi}(l)\|_{\HS}\leq C \exp{(-L \lambda_l^{\frac{1}{s\nu}})}
\quad \textrm{ for all } \; l\geq 0.
$$ 
In particular, for $s=1$, we recover the characterisation of analytic functions in
\eqref{EQ:analytic}.
\end{corollary}

In Corollary \ref{COR:Gevrey} we assume that $X$ and $E$ are analytic in order to
interpret the space $\gamma^{s}(X)$ locally as a Gevrey space, see
Theorem \ref{PROP:Komatsu}, (v). 

We now turn to the eigenfunction expansions of the corresponding spaces of
ultradistributions.

\begin{defi} 
The space ${\Gamma^{\prime}_{\{M_k\}}}(X)$ is the set of all linear forms $u$ on 
$\Gamma_{\{M_k\}}(X)$ such that for every $\epsilon>0$ there exists 
$C_{\epsilon}>0$ such that
$$
|u(\phi)|\leq C_{\epsilon} \sup_{\alpha} \epsilon^{|\alpha|}
M^{-1}_{\nu |\alpha|}\sup_{x\in X}|E^{|\alpha|}\phi(x)|
$$ 
holds for all $\phi\in \Gamma_{\{M_k\}}(X)$.
\end{defi}
We can define the Fourier coefficients of such $u$ by
$$
\widehat{u}(e^{k}_{l}):=u(\overline{e^k_l}) \;\textrm{ and }\;
\widehat{u}(l):=\widehat{u}(e_l):=\left[u(\overline{e^k_l})\right]_{k=1}^{d_l}.
$$

\begin{theorem}\label{THM:ultra}
Assume conditions {\rm (M.0), (M.1), (M.2).} 
We have $u\in \Gamma^{\prime}_{\{M_{k}\}}(X)$
if and only if for every $L>0$ there exists $K=K_{L}>0$ such that 
$$
\|\widehat{u}(l)\|_{\HS}\leq K\exp\left(M(L\lambda_l^{1/\nu})\right)
$$
holds for all $l\in\mathbb N$.
\end{theorem}

The spaces of ultradifferentiable functions in Definition \ref{DEF:Komatsu} can be
viewed as the spaces of Roumieu type. With natural modifications
the results remain true for spaces of 
Beurling type. We summarise them below. We will not give complete proofs but
can refer to \cite{Dasgupta-Ruzhansky:BSM} for details of such modifications
in the context of compact Lie groups.

The class $\Gamma_{(M_k)}(X)$ is the space of $C^{\infty}$ functions $\phi$ on $X$ 
such that for every $h>0$ there exists $C_{h}>0$ such that we have
\begin{equation}\label{EQ:def-Komatsu-L2b}
\|E^{k}\phi\|_{L^{2}(X)}\leq C_{h} h^{\nu k}M_{\nu k},~~~~~k=0,1,2,\ldots
\end{equation}
The counterpart of Theorem \ref{PROP:Komatsu}  holds for this class as well,
and we have

\begin{theorem}\label{THM:Komatsu-b}
Assume conditions {\rm (M.0), (M.1), (M.2).} 
We have $\phi\in{\Gamma_{(M_k)}(X)}$ if and only if for every $L>0$ 
there exists $C_{L}>0$ such that 
$$
\|\widehat{\phi}(l)\|_{\HS}\leq C_{L}\exp\{-M(L \lambda_l^{\onu} )\}
\quad \textrm{ for all } \; l\geq 1.
$$ 
For the dual space and for the $\alpha$-dual, 
the following statements are equivalent
\begin{itemize}
\item[(i)]  $v\in \Gamma^{\prime}_{(M_{k})}(X);$
\item[(ii)]  $v\in [\Gamma_{(M_{k})}(X)]^{\wedge};$
\item[(iii)] there exists $L>0$ such that we have 
$$
\sum_{l=1}^{\infty}
\exp \left(-M(L \lambda_l^{\onu})\right)\|v_l\|_{\HS}<\infty;
$$
\item[(iv)] there exist $L>0$ and $K>0$ such that 
$$
\|v_l\|_{\HS}\leq K\exp\left(M(L\lambda_l^{\onu})\right)
$$
holds for all $l\in\mathbb N$.
\end{itemize}
\end{theorem}
The proof of Theorem \ref{THM:Komatsu-b} is similar to the proof of the corresponding
results for the spaces $\Gamma_{\{M_k\}}(X)$, and we omit the repetition. The only
difference is that we need to use the K\"othe theory of sequence spaces at one
point, but this can be done analogous to \cite{Dasgupta-Ruzhansky:BSM}, so we 
may omit the details.
Finally, we note that given the characterisation of $\alpha$-duals, one can readily
proof that they are {\em perfect}, namely, that 
\begin{equation}\label{EQ:perfect}
\left[\Gamma_{\{M_k\}}(X)\right]=\left(\left[\Gamma_{\{M_k\}}(X)\right]^{\wedge}\right)^{\wedge}
\;\textrm{ and }\;
\left[\Gamma_{(M_k)}(X)\right]=\left(\left[\Gamma_{(M_k)}(X)\right]^{\wedge}\right)^{\wedge},
\end{equation}
see Definition \ref{DEF:alpha-dual} and condition \eqref{EQ:alpha-dual} for
their definition.
Again, once we have, for example, Theorem \ref{THM:Komatsu} and 
Theorem \ref{THM:alpha-duals}, the proof of \eqref{EQ:perfect} is purely
functional analytic and can be done almost identically to that in
\cite{Dasgupta-Ruzhansky:BSM}, therefore we will omit it.

\section{Proofs}
\label{SEC:proofs}

First we prove Theorem \ref{PROP:Komatsu} clarifying the definition of the
class $\Gamma_{\{M_{k}\}}(X)$. In the proof as well as in further proofs the
following estimate will be useful:
\begin{equation}\label{EQ:Weyl}
\|e_{l}^{j}\|_{L^{\infty}(X)}\leq C\lambda_{l}^{\frac{n-1}{2\nu}}
\textrm{ for all } l\geq 1.
\end{equation}
This estimate follows, for example, from the local Weyl law
\cite[Theorem 5.1]{Hormander:spectral-function-AM-1968}, see also 
\cite[Lemma 8.5]{Delgado-Ruzhansky:invariant}.

\begin{proof}[Proof of Theorem \ref{PROP:Komatsu}]
{\bf (i)} The statement would follow if we can show that for 
$E_{1}, E_{2}\in\Psi^{\nu}_{e}(X)$ there is a constant $A>0$ such that
\begin{equation}\label{EQ:ell1}
\|E_{1}^{k}\phi\|_{L^{2}(X)}\leq A^{k} k! \|E_{2}^{k}\phi\|_{L^{2}(X)}
\end{equation}
holds for all $k\in\N0$ and all $\phi\in C^{\infty}(X)$.
The estimate \eqref{EQ:ell1} follows from the fact that the pseudo-differential operator
$E_{1}^{k}\circ E_{2}^{-k}\in\Psi^{0}_{e}(X)$ is bounded on $L^{2}(X)$
(with $E_{2}^{-k}$ denoting the parametrix for $E_{2}^{k}$),
and by the Calderon-Vaillancourt theorem
its operator norm can be estimated by $A^{k}$ for some constant $A$
depending only on finitely many derivatives of symbols of $E_{1}$ and $E_{2}$.

{\bf (ii)} The equivalence between \eqref{EQ:def-Komatsu-Linf}
and \eqref{EQ:def-Komatsu-L2}
follows by embedding theorems but we give a
short argument for it in order to keep a more precise track of the appearing constants.
First we note that \eqref{EQ:def-Komatsu-Linf}
implies \eqref{EQ:def-Komatsu-L2} with a uniform constant
in view of the continuous embedding $L^{\infty}(X)\hookrightarrow L^{2}(X)$.
Conversely, suppose we have \eqref{EQ:def-Komatsu-L2}.
Let $\phi\in \Gamma_{\{M_{k}\}}(X)$. 
Then using \eqref{EQ:Weyl} we can estimate
\begin{eqnarray*}
\|\phi\|_{L^{\infty}(X)}&=& 
\|\sum_{j=0}^{\infty}\sum_{k=1}^{d_j} \widehat{\phi}(j,k)e_{j}^{k}\|_{L^{\infty}(X)}\nonumber\\
&\leq& \sum_{j=0}^{\infty}\sum_{k=1}^{d_j}
|\widehat{\phi}(j,k)| \, \|e_{j}^{k}\|_{L^{\infty}(X)}\nonumber\\
&\leq& C\|\widehat{\phi}(0)\|_{\HS}+C\sum_{j=1}^{\infty}\sum_{k=1}^{d_j}
|\widehat{\phi}(j,k)| \lambda_{j}^{\frac{n-1}{2\nu}}\nonumber\\&\leq& 
C\|\phi\|_{L^{2}(X)}+C\left(\sum_{j=1}^{\infty}\sum_{k=1}^{d_j}|\widehat{\phi}(j,k)| 
\lambda_{j}^{2\ell}\right)^{1/2}\left(\sum_{j=0}^{\infty}\sum_{k=1}^{d_j}
\lambda_{j}^{\frac{n-1}{\nu}-2\ell}\right)^{1/2},
\end{eqnarray*}
where we take $\ell$ large enough so that the very last sum converges,
see \eqref{EQ:dl2}.
This implies
\begin{multline}\label{EQ:aux1}
\|\phi\|_{L^{\infty}(X)}\leq C\|\phi\|_{L^{2}(X)}+ C^{\prime}\left(\sum_{j=1}^{\infty}\sum_{k=1}^{d_j}
|\widehat{\phi}(j,k)| \lambda_{j}^{2\ell}\right)^{1/2} \\
\leq C^{\prime}(\|\phi\|_{L^{2}(X)}+\|E^{\ell}\phi\|_{L^{2}(X)})
\end{multline}
by Plancherel's formula.
We note that \eqref{EQ:aux1} follows, in principle, also from the local Sobolev
embedding due to the ellipticity of $E$, however the proof above provides us
with a uniform constant.
Using \eqref{EQ:aux1} and (M.1) we can estimate
\begin{eqnarray*}
||E^{m}\phi||_{L^{\infty}(X)} 
&\leq& C||E^{m}\phi||_{L^{2}(X)}+C^{\prime}||E^{l+m}\phi||_{L^{2}(X)}\nonumber\\
&\leq& Ch^{\nu m}M_{\nu m}+C^{\prime \prime}h^{\nu(l+m)}M_{\nu(l+m)}\nonumber\\
&\leq& Ch^{\nu m}M_{\nu m}+ C^{\prime\prime\prime}h^{\nu(l+m)}h H^{\nu(l+m)-1}M_{\nu(l+m)-1}\nonumber\\
&\leq& {...}  \nonumber\\
&\leq& Ch^{\nu m}M_{\nu m}+ C^{\prime\prime\prime}h^{\nu(l+m)}h^{\nu l}H^{B_l+\nu lm}M_{\nu m}\nonumber\\
&\leq& C_{l} A^{\nu m} M_{\nu m}\nonumber
\end{eqnarray*}
for some $A$ independent of $m$, yielding \eqref{EQ:def-Komatsu-Linf}.

\medskip
{\bf (iv)} We note that the proof as in (ii) also shows the equivalence of \eqref{EQ:def-vf}
and \eqref{EQ:def-vf2}. Moreover, once we have condition \eqref{EQ:def-vf}, the 
statement {\bf (v)} follows by using $M_{k}\geq Ck!$ and the chain rule.

\medskip
{\bf (iii)}
Given properties (ii) and (iv), we need to show that
\eqref{EQ:def-Komatsu-L2} or \eqref{EQ:def-Komatsu-Linf}
are equivalent to \eqref{EQ:def-vf} or to \eqref{EQ:def-vf2}.
Using property (i), we can take $E=\sum_{j=1}^{N} \partial_{j}^{2}.$

To prove that \eqref{EQ:def-vf} implies \eqref{EQ:def-Komatsu-Linf} we use
the multinomial theorem\footnote{But we use it in the form 
adapted to the noncommutativity of vector fields, namely, although the coefficients are
all equal to one in the noncommutative form, the multinomial coefficient appears once
we make a choice for $\alpha=(\alpha_{1},\ldots,\alpha_{N}).$},
with the notation for multi-indices as in (iii).
With $Y_{j}\in\{\partial_{1},\ldots,\partial_{N}\}$, $1\leq j\leq |\alpha|$, and $\nu=2$,
%using property (M.3) 
we can estimate
\begin{eqnarray*}\label{EQ:estLG}
|(\sum_{j=1}^{N} \partial_{j}^{2})^{k}\phi(x)|&\leq& C\sum_{|\alpha|= k} \frac{k!}{\alpha!}\left| Y_1^{2}\ldots
Y_{|\alpha|}^{2}\phi(x)\right|\nonumber\\
&\leq& C \sum_{|\alpha|= k}\frac{k!}{\alpha!}A^{2|\alpha|}
M_{2|\alpha|}
\nonumber\\
&\leq& C A^{2k}M_{2k}
\sum_{|\alpha|= k}\frac{k! N^{|\alpha|}}{|\alpha|!} \nonumber\\
&\leq& C_{1} A^{2k}M_{2k} N^{k}  2^{k} \nonumber\\
&\leq& C_2 A_{1}^{2k} M_{2k},
\end{eqnarray*}
with $A_{1}=2NA$,
implying \eqref{EQ:def-Komatsu-Linf}.

Conversely, we argue that  \eqref{EQ:def-Komatsu-L2} implies \eqref{EQ:def-vf2}.
We write 
$\paal=P_{\alpha}\circ E^{k}$ with
$P_{\alpha}=\paal\circ E^{-k}.$
Here and below, in order to use precise calculus of pseudo-differential operators we 
may assume that we work on the space $L^{2}(X)\backslash H_{0}.$
The argument similar to that of (i) implies that there is a constant $A>0$ such that
$\|P_{\alpha}\phi\|_{L^{2}(X)}\leq A^{k} \|\phi\|_{L^{2}(X)}$ for all
$|\alpha|\leq \nu k$.
Therefore, we get
\begin{equation*}
\|\partial^{\alpha}\phi\|_{L^2}= \|P_{\alpha}\circ E^{k}\phi\|_{L^{2}}
\leq C A^{k} \| E^{k}\phi\|_{L^{2}}
\leq C' A^{k} h^{\nu k} M_{\nu k}
\leq C' A_1^{\nu k} M_{\nu k},
\end{equation*}
where we have used the assumption \eqref{EQ:def-Komatsu-L2}, and with $C'$ and 
$A_{1}=A^{1/\nu} h$ independent
of $k$ and $\alpha$. This completes the proof of (iii) and of the theorem.
\end{proof}

\begin{proof}[Proof of Theorem \ref{THM:Komatsu}]
{\bf ``Only if'' part.}
Let $\phi \in \Gamma_{\{M_k\}}(X).$ By the Plancherel formula
\eqref{EQ:Plancherel} we have
\begin{equation}\label{EQ:est1}
\|E^{m}\phi\|_{L^2(X)}^{2}
=\sum_{j}{\|\widehat{E^m\phi}(j)\|^{2}_{\HS}}
=\sum_{j} \|\lambda_{j}^{m}\widehat{\phi}(j)\|^{2}_{\HS}
=\sum_{j} \lambda_j^{2m} \|\widehat{\phi}(j)\|^{2}_{\HS}.
\end{equation}
Now since
$\|E^{m}\phi\|_{L^{2}(X)}\leq Ch^{\nu m}M_{\nu m}$, from \eqref{EQ:est1} we get
$$
\lambda_j^{m}\|\widehat{\phi}(j)\|_{\HS}\leq Ch^{\nu m}M_{\nu m}
$$ 
which implies
\begin{equation}\label{EQ:est2}
\|\widehat{\phi}(j)\|_{\HS}\leq Ch^{\nu m}M_{\nu m} \lambda_j^{-m}\quad
\textrm{ for all } \; j\geq 1.
\end{equation}
Now, from the definition of $M(r)$ it follows that  
\begin{equation}\label{EQ:est3i}
\inf_{k\in\mathbb N}r^{-\nu k}M_{\nu k}=\exp\left(-M(r)\right), \quad r>0.
\end{equation}
Indeed, this identity follows by writing 
$$
{\exp(M(r))}=\exp\left({\sup_k}\log\frac{r^{\nu k}}{M_{\nu k}}\right)
= \sup_k\left(\exp{\log}\frac{r^{\nu k}}{M_{\nu k}}\right)
= \sup_k\left(\frac{r^{\nu k}}{M_{\nu k}}\right)
$$
and using the identity
$\inf_{k}r^{-k}=\frac{1}{\sup_{k}r^{k}}.$
Setting $r=\frac{\lambda_j^{\onu}}{h}$, from \eqref{EQ:est2} and \eqref{EQ:est3i} we can
estimate
\begin{multline*}
\|\widehat{\phi}(j)\|_{\HS}
\leq C\inf_{m\geq 1}\left\{\frac{h^{\nu m}}{{\lambda_j^m}}M_{\nu m}\right\}
= C\inf_{m\geq 1}r^{-\nu m}M_{\nu m} \\
= C\exp\left(-M\left(\frac{\lambda_j^{\onu}}{h}\right)\right)
= C\exp\left(-M\left(L\lambda_j^{\onu}\right)\right),
\end{multline*}
where $L=h^{-1}.$

\medskip
\textbf{``If'' Part.}
Let $\phi \in C^{\infty}(X)$ be such that 
$$
\|\widehat{\phi}(j)\|_{\HS}\leq 
C\exp\left(-M\left(L \lambda_j^{\onu}\right)\right)
$$
holds for all $j\geq 1$.
Then by Plancherel's formula we have
\begin{eqnarray}
\|E^{m}\phi\|_{L^{2}(X)}^{2}&=& \sum_{j=0}^{\infty}\lambda_{j}^{2m}
\|\widehat{\phi}(j)\|^{2}_{\HS}\nonumber\\
&\leq& C\sum_{j=1}^{\infty}\lambda_{j}^{2m}
\exp\left(-M(L\lambda_j^{\onu})\right)\exp\left(-M(L\lambda_j^{\onu})\right).
\end{eqnarray}
Now we observe that
\begin{eqnarray}
\lambda_{j}^{2m}\exp\left(-M(L\lambda_j^{\onu})\right)\leq
\frac{\lambda_{j}^{2m}}{\sup_{p\in\mathbb N}
\frac{L^{\nu p}{\lambda_j^{p}}}{M_{\nu p}}}
\leq \frac{\lambda_{j}^{2m}}{L^{\nu p}{\lambda_j^{p}}}
{M_{\nu p}}
\end{eqnarray}
for any $p\in\mathbb N$.
Using this with $p=2m+1$, we get
$$
\lambda_{j}^{2m}\exp\left(-M(L\lambda_j^{\onu})\right)
\leq {\frac{1}{L^{\nu(2m+1)}}\left(\frac{\lambda_{j}^{2m}}
{{\lambda}_j^{2m+1}}\right)}M_{\nu(2m+1)}.
$$
Then, by this and property (M.1) of the sequence $\{M_{k}\}$, we get

\begin{eqnarray}\label{EQ:est3}
\|E^{m}\phi\|_{L^{2}(X)}^{2} &\leq& C \sum_{j=1}^{\infty}{\frac{1}{L^{\nu(2m+1)}}
\left(\frac{\lambda_{j}^{2m}}{{\lambda}_j^{2m+1}}\right)}M_{\nu(2m+1)}
\exp\left(-M(L\lambda_j^{\onu})\right)\nonumber\\
&\leq& C_1AH^{2\nu m}M_{2\nu m}\sum_{j=1}^{\infty}{\lambda}_j^{-1}
\exp\left(-M(L\lambda_j^{\onu})\right),
\end{eqnarray}
for some $A,H>0$.
Now we note that for all $j\geq 1$ we have

\begin{eqnarray}
{\lambda}_j^{-1}\exp\left(-M(L\lambda_j^{\onu})\right)\leq
\frac{{\lambda}_j^{-1}}{{\lambda}_j^p}\frac{M_{\nu p}}{L^{\nu p}}
=\frac{M_{\nu p}}{L^{\nu p}}\frac{1}{{\lambda}_j^{p+1}}.
\end{eqnarray} 
In particular, in view of \eqref{EQ:dl2}, for $p$ such that $p+1>n/\nu$ we obtain
$$
\sum_{j=1}^{\infty} \lambda_j^{-1}\exp\left(-M(L\lambda_j^{\onu}))\right)\leq 
\frac{M_{\nu p}}{L^{\nu p}}\sum_{j=1}^{\infty}
\frac{1}{{\lambda}_j^{p+1}}<\infty.
$$
This, (M.2) and \eqref{EQ:est3} imply
$$
\|E^{m}\phi\|_{L^{2}(X)}\leq 
\tilde{A} \tilde{H}^{\nu m}M_{\nu m}
$$ 
and hence $\phi\in \Gamma_{\{M_{k}\}}(X)$.
\end{proof}

Now we will check that Seeley's characterisation for analytic functions 
in \cite{see:exp}
follows from our theorem.
\begin{proof}[Proof of Corollary \ref{COR:Gevrey}]
The first part of the statement is a direct consequence of Theorem \ref{THM:Komatsu}
and \eqref{EQ:M-Gevrey}, so we only 
have to prove \eqref{EQ:analytic}.
Let $X$ be a compact manifold and $E$ an analytic, elliptic, positive differential 
operator of order $\nu.$ Let $\{\phi_k\}$ and $\{\lambda_k\}$ be respectively the eigenfunctions and eigenvalues of $E,$ i.e. $E\phi_k=\lambda_k\phi_k.$
As mentioned in the introduction, Seeley showed in \cite{see:exp}
that a $C^\infty$ function $f=\sum_{j}f_j\phi_j$ is analytic if and only if
there is a constant $C>0$ such that for all $k\geq 0$ we have
$$
\sum_{j} \lambda_j^{2k}|f_j|^2\leq ((\nu k)!)^2C^{2k+2}.
$$ 
By Plancherel's formula this means that
$$
\|E^{k}f\|^{2}_{L^2(X)}=\sum_{j} \lambda_j^{2k}|f_j|^{2}\leq ((\nu k)!)^2C^{2k+2}.
$$
For the class of analytic functions we can take
$M_{k}=k!$ in Definition \ref{DEF:Komatsu}, and then by
Theorem \ref{THM:Komatsu}
we conclude that $f$ is analytic if and only if
$$
\|\widehat{f}(j)\|_{\HS}\leq 
C\exp(-L\lambda_{j}^{\onu})
$$ 
or to
$$
|f_{j}|\leq 
C^{\prime}\exp(-L^{\prime}\lambda_{j}^{\onu}),
$$ 
with 
$M(r)=\sup_{p}\log\frac{r^{\nu p}}{(\nu p)!}\simeq r$ in view of \eqref{EQ:M-Gevrey}. 
This implies \eqref{EQ:analytic} and hence also Seeley's result \cite{see:exp} that 
$f=\sum_{j}f_j\phi_j$ is analytic if and only if 
the sequence $\{A^{\lambda_j^{\onu}}f_j\}$ is bounded for some $A>1$.
\end{proof}

\section{$\alpha$-duals}
\label{SEC:alpha}

In this section we characterise the $\alpha$-dual of the space
$\Gamma_{\{M_{k}\}}(X)$. This will be instrumental in proving the characterisation
for spaces of ultradistributions in Theorem \ref{THM:ultra}.

\begin{defi}\label{DEF:alpha-dual}
The $\alpha$-dual of the space $\Gamma_{\{M_{k}\}}(X)$ of ultradifferentiable functions, 
denoted by $[\Gamma_{\{M_{k}\}}(X)]^{\wedge},$ is defined as 
$$
\left\{v=(v_l)_{l\in \N0}: \sum_{l=0}^{\infty}\sum_{j=1}^{d_l}|(v_l)_{j}|
|\widehat{\phi}(l,j)|<\infty, v_{l}\in {\mathbb{C}}^{d_l},
\textrm{ for all } \phi\in \Gamma_{\{M_{k}\}}(X) \right\}.
$$
We will also write $v(l,j)=(v_{l})_{j}$ and
$\|v_{l}\|_{\HS}=(\sum_{j=1}^{d_{l}}|v(l,j)|^{2})^{1/2}.$
\end{defi}
It will be useful to have the definition of
the second dual $\left([\Gamma_{\{M_{k}\}}(X)]^{\wedge}\right)^{\wedge}$
as the space of $w=(w_l)_{l\in \N0}$, $w_{l}\in {\mathbb{C}}^{d_l}$ such that
\begin{equation}\label{EQ:alpha-dual}
\sum_{l=0}^{\infty}\sum_{j=1}^{d_l}|(w_l)_{j}| |(v_l)_{j}|
<\infty\quad
\textrm{ for all } v\in [\Gamma_{\{M_{k}\}}(X)]^{\wedge}.
\end{equation}
We have the following characterisations of the $\alpha$-duals.

\begin{theorem}\label{THM:alpha-duals}
Assume conditions {\rm (M.0), (M.1)} and {\rm (M.2)}.
The following statements are equivalent.
\begin{itemize}
\item[(i)]  $v\in [\Gamma_{\{M_{k}\}}(X)]^{\wedge};$
\item[(ii)] for every $L>0$ we have 
$$
\sum_{l=1}^{\infty}
\exp \left(-M(L \lambda_l^{\onu})\right)\|v_l\|_{\HS}<\infty;
$$
\item[(iii)] for every $L>0$ there exists $K=K_{L}>0$ such that 
$$
\|v_l\|_{\HS}\leq K\exp\left(M(L\lambda_l^{\onu})\right)
$$
holds for all $l\in\mathbb N$.
\end{itemize}
\end{theorem}

\begin{proof} 
{\bf (i) $\Longrightarrow$ (ii).}
Let $v\in [\Gamma_{\{M_{k}\}}(X)]^{\wedge}$, $L>0$, and let $\phi\in C^{\infty}(X)$ be such that
$$
\widehat{\phi}(l,j)=e^{-M(L\lambda_l^{\onu})}.
$$
We claim that $\phi\in\Gamma_{\{M_k\}}(X).$
First, using \eqref{EQ:dl1}, for some $q$ we have
\begin{eqnarray}\label{EQ:est4}
\|\widehat{\phi}(l)\|_{\HS}= d_{l}^{1/2}e^{-M(L\lambda_l^{\onu})}
\leq C\lambda_l^{q}e^{-\frac{1}{2}M\left(L\lambda_l^{\onu}\right)}
e^{-\frac{1}{2}M\left(L\lambda_l^{\onu}\right)}
\end{eqnarray}
for all $l\geq 1$.
Estimates \eqref{EQ:est-ev} and \eqref{EQ:est4} imply
that $\lambda_l^{q}e^{-\frac{1}{2}M\left(L\lambda_l^{\onu}\right)}\leq C$ and hence
$$
\|\widehat{\phi}(l)\|_{\HS}\leq 
C^{\prime}e^{-\frac{1}{2}M\left(L\lambda_l^{\onu}\right)}
$$
holds for all $l\geq 1$. The claim would follow if we can show that
\begin{equation}\label{EQ:exps}
e^{-\frac{1}{2}M\left(L\lambda_l^{\onu}\right)}\leq e^{-M\left(L_{2}\lambda_l^{\onu}\right)}
\textrm{ holds for } L_{2}=\frac{L}{\sqrt{A} H},
\end{equation}
where $A$ and $H$ are constants in the condition (M.2).
Now, substituting $p=2q$, we note that 
\begin{equation}\label{EQ:exp1}
e^{-\frac{1}{2}M\left(L\lambda_l^{\onu}\right)}=\inf_{p\in\mathbb N}
\frac{M^{1/2}_{\nu p}}{L^{\nu p/2}{\lambda_l^{p/2}}}
\leq 
\inf_{q\in\mathbb N}\frac{M^{1/2}_{2\nu q}}{L^{\nu q}{\lambda_l^{q}}}.
\end{equation}
Using property (M.2) we can estimate
$$M_{2\nu q}\leq A H^{2\nu k}M_{\nu q}^{2}.$$
This and \eqref{EQ:exp1} imply
$$
e^{-\frac{1}{2}M\left(L\lambda_l^{\onu}\right)}\leq 
\frac{M_{\nu q}}{L_2^{\nu q}{\lambda_l^{q}}},
$$ 
where $L_{2}=\frac{L}{\sqrt{A} H}.$
Taking infimum in $q\in\mathbb N$, we obtain
$$
e^{-\frac{1}{2}M\left(L\lambda_l^{\onu}\right)}\leq
\inf_{q\in \mathbb N} 
\frac{M_{\nu q}}{L_2^{\nu q}{\lambda_l^{q}}}= 
e^{-M(L_2 \lambda_{l}^{\onu})}. 
$$
Therefore, we get the estimate
$$
\|\widehat{\phi}(l)\|_{\HS}\leq C^{\prime}\exp\left(-M(L_2 \lambda_l^{\onu})\right),
$$
which means that $\phi\in\Gamma_{\{M_k\}}(X)$ by Theorem \ref{THM:Komatsu}.
Finally, this implies that 
\begin{eqnarray*}
\sum_{l}e^{-M(L\lambda_l^{\onu})}\|v_l\|_{\HS}&\leq& \sum_{l}
\sum_{j=1}^{d_{l}} e^{-M(L \lambda_l^{\onu})} |v_{l}(j)|\nonumber\\
&=& \sum_{l}\sum_{j=1}^{d_l}|\widehat{\phi}(l,j)| |v(l,j)|<\infty
\end{eqnarray*}
is finite by property (i), implying (ii).

\medskip
\textbf{(ii) $\Longrightarrow$ (i).}
Let $\phi\in\Gamma_{\{M_k\}}(X).$ Then by Theorem \ref{THM:Komatsu}
there exists $L>0$ such that
$$
\|\widehat{\phi}(l)\|_{\HS}\leq 
C\exp\left(-M(L\lambda_l^{\onu})\right).
$$
Then we can estimate 
\begin{eqnarray*}
\sum_{l=0}^{\infty}
\sum_{j=1}^{d_l}
|(v_l)_{j}| |\widehat{\phi}(l,j)| &\leq& \sum_{l=0}^{\infty}
\|v_l\|_{\HS} \|\widehat{\phi}(l)\|_{\HS}\nonumber\\
&\leq& C\sum_{l=0}^{\infty}
\exp \left(-M(L\lambda_l^{\onu})\right) \|v_l \|_{\HS}<\infty
\end{eqnarray*}
is finite by the assumption (ii).
This implies
$v\in [\Gamma_{\{M_k\}}(X)]^{\wedge}.$

 \medskip 
 \textbf{(ii) $\Longrightarrow$ (iii).}
 We know that for every $L>0$ we have
 $$
 \sum_{l}\exp \left(-M(L\lambda_l^{\onu})\right)\|v_l\|_{\HS}<\infty.
 $$
 Consequently, there exists $K_L$ such that 
 $$
 \exp \left(-M(L\lambda_l^{\onu})\right)\|v_l\|_{\HS}\leq K_L
 $$ 
 holds for all $l$,
 which implies (iii).

\medskip 
{\bf (iii) $\Longrightarrow$ (ii).}
Let $L>0$. Let us define $L_{2}$ as in 
\eqref{EQ:exps}. If $v$ satisfies (iii) this means, in particular,
that there exists $K=K_{L_{2}}>0$ such that
\begin{equation}\label{EQ:vl1}
\|v_l\|_{\HS}\leq 
K\exp\left(M(L_{2}\lambda_l^{\onu})\right).
\end{equation}
%We observe that
%\begin{equation}\label{EQ:Mcomp}
%e^{M(r/2)}=\sup_{k\in\mathbb N} \frac{r^{\nu k}}{2^{\nu k}M_{\nu k}}\leq
%\frac12 e^{M(r)},\quad r>0.
%\end{equation}
We also note that by \eqref{EQ:exp1} we have
$$
\exp\left(-\frac{1}{2}M(L\lambda_l^{\onu})\right)\leq 
\frac{M_{\nu p}^{1/2}}{L^{\nu p/2}\lambda_l^{p/2}}
\quad \textrm{ for all } p\in\mathbb N.
$$
From this, \eqref{EQ:exps} and \eqref{EQ:vl1} we conclude
 \begin{eqnarray*} 
 & & \sum_{l=0}^{\infty}
\exp \left(-M(L \lambda_l^{\onu})\right)\|v_l\|_{\HS}  \\
&\leq&
\sum_{l=0}^{\infty}
\exp \left(-\frac12 M(L \lambda_l^{\onu})\right)\exp\left(-\frac12 M(L\lambda_l^{\onu})\right) 
\|v_l\|_{\HS} \\
&\leq&
\sum_{l=0}^{\infty}
\exp \left(-\frac12 M(L \lambda_l^{\onu})\right)\exp\left(- M(L_{2}\lambda_l^{\onu})\right) 
\|v_l\|_{\HS} \\
  &\leq& K\sum_{l=0}^{\infty}\exp\left(-\frac{1}{2}M(L\lambda_l^{\onu})\right)\nonumber\\
 &\leq& K+K\sum_{l=1}^{\infty}\frac{M_{\nu p}^{1/2}}{L^{\nu p/2}\lambda_l^{p/2}}\nonumber\\
 &\leq& K+C_{p}\sum_{l=1}^{\infty} \frac{1}{\lambda_l^{p/2}}<\infty
 \end{eqnarray*}
 is finite provided we take $p$ large enough in view of \eqref{EQ:dl2}.
\end{proof}

\section{Ultradistributions}
\label{SEC:ultradistributions}

In this section we prove that the spaces of ultradistributions and $\alpha$-duals
coincide. Together with Theorem \ref{THM:alpha-duals} this implies 
Theorem \ref{THM:ultra}.

\begin{theorem}\label{THM:distribution}
Assume conditions {\rm (M.0), (M.1)} and {\rm (M.2)}.
We have $v\in {\Gamma^{\prime}_{\{M_k\}}}(X)$ if and only if 
$v\in [\Gamma_{\{M_k\}}(X)]^{\wedge}.$
\end{theorem}

\begin{proof} 
\textbf{``If'' Part.}
Let $v\in [\Gamma_{\{M_k\}}(X)]^{\wedge}.$ 
For any $\phi\in \Gamma_{\{M_k\}}(X)$ let us 
define
$$
v(\phi):=\sum_{l=0}^{\infty}\widehat{\phi}(l)\cdot v_l=
\sum_{l=0}^{\infty}\sum_{j=1}^{d_{l}} \widehat{\phi}(l,j) v_l(j).
$$ 
Given $\phi\in \Gamma_{\{M_k\}}(X)$, by Theorem \ref{THM:Komatsu} there exist
$C>0$ and $L>0$ such that
$$
\|\widehat{\phi}(l)\|_{\HS}\leq 
Ce^{-M(L\lambda_l^{\onu})}.
$$
Consequently, by Theorem \ref{THM:alpha-duals}, we get
\begin{eqnarray*}
|v(\phi)|\leq \sum_{l}\|\widehat{\phi}(l)\|_{\HS} \|v_l\|_{\HS}
\leq C\sum_{l}e^{-M(L\lambda_l^{\onu})}\|v_l\|_{\HS}<\infty,
\end{eqnarray*}
which means that $v(\phi)$ is a well-defined linear functional. 
Next we check that $v$ is continuous.
Suppose $\phi_j\rightarrow\phi$ in $\Gamma_{\{M_k\}}(X)$ as $j\rightarrow\infty,$ 
which means that 
$$
\sup_{\alpha}\epsilon^{|\alpha|}M^{-1}_{\nu |\alpha|}\sup_{x\in X}
|E^{|\alpha|}\left(\phi_j(x)-\phi(x)\right)|\rightarrow 0
\textrm{ as } j\rightarrow\infty.
$$
It follows that 
$$
\|E^{|\alpha|}\left(\phi_j(x)-\phi(x)\right)\|_{L^{\infty}(X)}\leq 
C_j A^{|\alpha|}M_{\nu |\alpha|},
$$
where $C_j\rightarrow 0$ as $j\rightarrow \infty.$
From the proof of Theorem \ref{THM:Komatsu} it follows that 
$$
\|\widehat{\phi_j}(l)-\widehat{\phi}(l)\|_{\HS}\leq C_{j}^{\prime}e^{-M(L\lambda_l^{\onu})}
$$
with $C_{j}^{\prime}\to 0$.
Hence 
$$
|v\left(\phi_j-\phi\right)|\leq
\sum_{l}\|\widehat{\phi_j}(l)-\widehat{\phi}(l)\|_{\HS} \|v_l\|_{\HS} \leq
C_{j}^{\prime} \sum_{l}e^{-M(L\lambda_l^{\onu})} \|v_l\|_{\HS}
\rightarrow 0
$$ 
as $j\rightarrow\infty$. This implies $v\in  {\Gamma^{\prime}_{\{M_k\}}}(X).$

\medskip
\textbf{``Only if'' Part.}
Let $v\in  {\Gamma^{\prime}_{\{M_k\}}}(X).$ By definition, this implies
$$
|v(\phi)|\leq C_{\epsilon}\sup_{\alpha}\epsilon^{|\alpha|}
M^{-1}_{\nu |\alpha|}\sup_{x\in X}|E^{|\alpha|}\phi(x)|
$$ 
for all $\phi\in\Gamma_{\{M_k\}}(X).$ In particular,
\begin{eqnarray}
|v(\overline{e^{j}_l})|&\leq& C_{\epsilon}\sup_{\alpha}\epsilon^{|\alpha|}M^{-1}_{\nu|\alpha|}\sup_{x\in X}|E^{|\alpha|}\overline{e^{j}_{l}}(x)|\nonumber\\
&\leq& C_{\epsilon}\sup_{\alpha}\epsilon^{|\alpha|}M^{-1}_{\nu|\alpha|}{\lambda}_l^{|\alpha|}\sup_{x\in X}|\overline{e^{j}_{l}}(x)|\nonumber\\
&\leq&C_{\epsilon}\sup_{\alpha}
\frac{\epsilon^{|\alpha|}
{\lambda}_l^{|\alpha|+\frac{n-1}{2\nu}}}{M_{\nu|\alpha|}}.\nonumber
\end{eqnarray}
Here in the last line we used the estimate \eqref{EQ:Weyl}.
Consequently, we get
\begin{eqnarray}\label{EQ:estM}
|v(\overline{e^{j}_l})|\leq C_{\epsilon}
{\lambda}_l^{\frac{n-1}{2\nu}}
\sup_{\alpha}\frac{\epsilon^{|\alpha|}
{\lambda}_l^{|\alpha|}}{M_{\nu|\alpha|}}\leq
C_{\epsilon}^{\prime}
\sup_{\alpha}\frac{\epsilon^{|\alpha|}
{\lambda}_l^{|\alpha|+k}}{M_{\nu |\alpha|}}
\end{eqnarray}
with $k:=[\frac{n-1}{2\nu}]+1.$
By property (M.1) of the sequence $\{M_{k}\}$ we can estimate
\begin{eqnarray*}
M_{\nu(|\alpha|+k)}&\leq& A H^{\nu(|\alpha|+k)-1}M_{\nu(|\alpha|+k)-1}\nonumber\\
&\leq& A^{2} H^{2\nu(|\alpha|+k)-1-2}M_{\nu(|\alpha|+k)-2}\nonumber\\
&\vdots&\\
&\leq& A^{\nu k}H^{\nu k|\alpha|}H^{f(k)}M_{\nu|\alpha|},\nonumber
\end{eqnarray*}
for some $f(k)=f(\nu,k)$ independent of $\alpha$.
This implies
$$
M^{-1}_{\nu |\alpha|}\leq A^{\nu k}H^{f(k)}H^{\nu k|\alpha|}M^{-1}_{\nu(|\alpha|+k)}.
$$
This and \eqref{EQ:estM} imply
\begin{multline}
|v(\overline{e^{j}_l})|
\leq C^{\prime}_{\epsilon}\
\epsilon^{-k}A^{\nu k} H^{f(k)}\sup_{\alpha}
\frac{\epsilon^{|\alpha|+k}(H^{\nu k})^{|\alpha|}
\lambda_l^{|\alpha|+k}}{M_{\nu(|\alpha|+k})} \\
\leq C^{\prime}_{\epsilon,k,A}\sup_{\alpha}
\frac{(\epsilon^{1/\nu} H^{k})^{\nu(|\alpha|+k)}
\lambda_l^{|\alpha|+k}}{M_{\nu(|\alpha|+k)}}
\leq
C^{\prime}_{\epsilon,k,A} e^{M(L\lambda_{l}^{\onu})},
\end{multline}
with $L=\epsilon^{\onu} H^{k}$.
At the same time, it follows from \eqref{EQ:exps} that 
\begin{equation*}\label{EQ:exps2}
e^{M\left(L\lambda_l^{\onu}\right)} \leq
e^{\frac{1}{2}M\left(L_{3}\lambda_l^{\onu}\right)}
\textrm{ holds for } L=\frac{L_{3}}{\sqrt{A} H}.
\end{equation*}
This and \eqref{EQ:dl1}
for some $q$, and then \eqref{EQ:est-ev} imply 
$$
\|\widehat{v}(e_l)\|_{\HS}\leq C d_{l}^{1/2} e^{M(L\lambda_l^{\onu})}
\leq C \lambda_{l}^{q} e^{\frac{1}{2}M(L_{3}\lambda_l^{\onu})}
\leq C e^{M(L_{3}\lambda_l^{\onu})},
$$
that is 
$v\in \left[\Gamma_{\{M_k\}}(X)\right]^{\wedge}$
by Theorem \ref{THM:alpha-duals}.
\end{proof}

%\bibliographystyle{alphaabbr}
%\bibliography{bib-pdo-2014-10-3}

\begin{thebibliography}{DHPV04}

\bibitem[CKP07]{Carmichael-Kamiski-Pilipovic:BK}
R.~D. Carmichael, A.~Kami{{\'n}}ski, and S.~Pilipovi{{\'c}}.
\newblock {\em Boundary values and convolution in ultradistribution spaces},
  volume~1 of {\em Series on Analysis, Applications and Computation}.
\newblock World Scientific Publishing Co. Pte. Ltd., Hackensack, NJ, 2007.

\bibitem[DHPV04]{Delcroix-Hasler-Pilipovic:periodic}
A.~Delcroix, M.~F. Hasler, S.~Pilipovi{{\'c}}, and V.~Valmorin.
\newblock Embeddings of ultradistributions and periodic hyperfunctions in
  {C}olombeau type algebras through sequence spaces.
\newblock {\em Math. Proc. Cambridge Philos. Soc.}, 137(3):697--708, 2004.

\bibitem[DR14a]{Dasgupta-Ruzhansky:BSM}
A.~Dasgupta and M.~Ruzhansky.
\newblock Gevrey functions and ultradistributions on compact {L}ie groups and
  homogeneous spaces.
\newblock {\em Bull. Sci. Math.}, 138(6):756--782, 2014.

\bibitem[DR14b]{Delgado-Ruzhansky:invariant}
J.~Delgado and M.~Ruzhansky.
\newblock Fourier multipliers, symbols and nuclearity on compact manifolds.
\newblock {\em arXiv:1404.6479}, 2014, to appear in J. Anal. Math..

\bibitem[DR14c]{Delgado-Ruzhansky:JFA-2014}
J.~Delgado and M.~Ruzhansky.
\newblock Schatten classes on compact manifolds: kernel conditions.
\newblock {\em J. Funct. Anal.}, 267(3):772--798, 2014.

\bibitem[Gor82]{Gorbachuk:Komatsu-UMJ-1982}
V.~I. Gorbachuk.
\newblock Fourier series of periodic ultradistributions.
\newblock {\em Ukrain. Mat. Zh.}, 34(2):144--150, 267, 1982.

\bibitem[GR15]{Garetto-Ruzhansky:wave-eq}
C.~Garetto and M.~Ruzhansky.
\newblock Wave equation for sums of squares on compact {L}ie groups.
\newblock {\em J. Differential Equations}, 258(12):4324--4347, 2015.

\bibitem[H{\"o}r68]{Hormander:spectral-function-AM-1968}
L.~H{\"o}rmander.
\newblock The spectral function of an elliptic operator.
\newblock {\em Acta Math.}, 121:193--218, 1968.

\bibitem[Kom73]{Komatsu:ultra-I}
H.~Komatsu.
\newblock Ultradistributions. {I}. {S}tructure theorems and a characterization.
\newblock {\em J. Fac. Sci. Univ. Tokyo Sect. IA Math.}, 20:25--105, 1973.

\bibitem[Kom77]{Komatsu:ultra-II}
H.~Komatsu.
\newblock Ultradistributions. {II}. {T}he kernel theorem and ultradistributions
  with support in a submanifold.
\newblock {\em J. Fac. Sci. Univ. Tokyo Sect. IA Math.}, 24(3):607--628, 1977.

\bibitem[Kom82]{Komatsu:ultra-III}
H.~Komatsu.
\newblock Ultradistributions. {III}. {V}ector-valued ultradistributions and the
  theory of kernels.
\newblock {\em J. Fac. Sci. Univ. Tokyo Sect. IA Math.}, 29(3):653--717, 1982.

\bibitem[K{\"o}t69]{Kothe:BK-top-vector-spaces-I}
G.~K{\"o}the.
\newblock {\em Topological vector spaces. {I}}.
\newblock Translated from the German by D. J. H. Garling. Die Grundlehren der
  mathematischen Wissenschaften, Band 159. Springer-Verlag New York Inc., New
  York, 1969.

\bibitem[Ney69]{Neymark:1969}
M.~Neymark.
\newblock On the {L}aplace transform of functionals on classes of infinitely
  differentiable functions.
\newblock {\em Ark. Mat.}, 7:577--594 (1969), 1969.

\bibitem[Pet79]{Petzsche:MM-1979}
H.-J. Petzsche.
\newblock Die {N}uklearit{\"a}t der {U}ltradistributionsr{\"a}ume und der
  {S}atz vom {K}ern. {II}.
\newblock {\em Manuscripta Math.}, 27(3):221--251, 1979.

\bibitem[PV84]{Petzsche: MA-1984} 
H.-J. Petzsche and D. Vogt.
\newblock Almost analytic extension of ultradifferentiable functions and the boundary values of holomorphic functions.
\newblock{\em Math. Ann.} 267: 17-35, 1984.

\bibitem[PP14]{Pilipovic-Prangoski:Roumieu-MM-2014}
S.~Pilipovi{{\'c}} and B.~Prangoski.
\newblock On the convolution of {R}oumieu ultradistributions through the
  {$\epsilon$} tensor product.
\newblock {\em Monatsh. Math.}, 173(1):83--105, 2014.

\bibitem[PS01]{Pilipovic-Scarpalezos:PAMS-2001}
S.~Pilipovi{{\'c}} and D.~Scarpalezos.
\newblock Regularity properties of distributions and ultradistributions.
\newblock {\em Proc. Amer. Math. Soc.}, 129(12):3531--3537 (electronic), 2001.

\bibitem[Rod93]{Rodino:bk-Gevrey}
L.~Rodino.
\newblock {\em Linear partial differential operators in {G}evrey spaces}.
\newblock World Scientific Publishing Co., Inc., River Edge, NJ, 1993.

\bibitem[Rou63]{Roumieu:1962}
C.~Roumieu.
\newblock Ultra-distributions d{\'e}finies sur {$R^{n}$} et sur certaines
  classes de vari{\'e}t{\'e}s diff{\'e}rentiables.
\newblock {\em J. Analyse Math.}, 10:153--192, 1962/1963.

\bibitem[Rud74]{Rudin:bk-RandCanalysis-1974}
W.~Rudin.
\newblock {\em Real and complex analysis}.
\newblock McGraw-Hill Book Co., New York-D{\"u}sseldorf-Johannesburg, second
  edition, 1974.
\newblock McGraw-Hill Series in Higher Mathematics.

\bibitem[See65]{see:ex}
R.~T. Seeley.
\newblock Integro-differential operators on vector bundles.
\newblock {\em Trans. Amer. Math. Soc.}, 117:167--204, 1965.

\bibitem[See69]{see:exp}
R.~T. Seeley.
\newblock Eigenfunction expansions of analytic functions.
\newblock {\em Proc. Amer. Math. Soc.}, 21:734--738, 1969.

\bibitem[Tag87]{Taguchi:torus-YMJ-1987}
Y.~Taguchi.
\newblock Fourier coefficients of periodic functions of {G}evrey classes and
  ultradistributions.
\newblock {\em Yokohama Math. J.}, 35(1-2):51--60, 1987.

\end{thebibliography}
%\end{document}
\section*{Acknowledgement}
The authors would like to thank Jens Wirth for a discussion. The authors would also like to thank Jasson Vindas for discussions and comments.

\end{document}